\documentclass[12pt]{article}
\usepackage{amsmath, amsthm, amssymb}
\usepackage{mathrsfs}
\usepackage{color}
\usepackage{enumerate}
\usepackage{url}
\usepackage{hyperref} 
\usepackage{float}
\usepackage{tikz}
\input{pdfcolor}

\newcommand{\bproof}{\begin{proof}}
\newcommand{\eproof}{\end{proof}}

\newcommand{\GF}{{\rm GF}}

\newcommand{\C}{{\mathbb C}}
\newcommand{\PG}{{\rm PG}}

\newcommand{\Cay}{{\rm Cay}}
\newcommand{\Cl}{{\rm Cl}}

\newcommand{\calH}{\mathcal{H}}
\newcommand{\calO}{\mathcal{O}}

\numberwithin{equation}{section}
\newtheorem{thm}{Theorem}[section]
\newtheorem{lemma}[thm]{Lemma}
\newtheorem{cor}[thm]{Corollary}

\theoremstyle{definition}\newtheorem{defn}[thm]{Definition}
\newtheorem{example}[thm]{Example}
\newtheorem{question}[thm]{Question}
\newtheorem{rem}[thm]{Remark}

\allowdisplaybreaks[4]

\newcommand{\noid}{\vskip12pt\noindent}
\newcommand{\vs}{\vskip12pt}
\newcommand{\ZZ}{\mathbb{Z}}
\newcommand{\be}{\begin{enumerate}}
\newcommand{\ee}{\end{enumerate}}

\begin{document}

\title{Genuinely nonabelian partial difference sets\\
}
\author{
John Polhill\footnote{Commonwealth University, Bloomsburg, PA, {\tt jpolhill@commonwealthu.edu}},
James A. Davis\footnote{University of Richmond, VA 23173, {\tt jdavis@richmond.edu}},
Ken Smith\footnote{Huntsville, TX 77340, {\tt kenwsmith54@gmail.com}},
Eric Swartz\footnote{William \& Mary, Williamsburg, VA 23187, {\tt easwartz@wm.edu}}
}

\date{}

\maketitle

\begin{abstract}
Strongly regular graphs (SRGs) provide a fertile area of exploration in algebraic combinatorics, integrating techniques in graph theory, linear algebra, group theory, finite fields, finite geometry, and number theory. Of particular interest are those SRGs with a large automorphism group.
If an automorphism group acts regularly (sharply transitively) on the vertices of the graph, then we may identify the graph with a subset of the group, a partial difference set (PDS), which allows us to apply techniques from group theory to examine the graph. 
Much of the work over the past four decades has concentrated on abelian PDSs using the powerful techniques of character theory.
However, little work has been done on nonabelian PDSs. 
In this paper we point out the existence of \textit{genuinely nonabelian} PDSs, i.e., PDSs for parameter sets where a nonabelian group is the only possible regular automorphism group. We include methods for demonstrating that abelian PDSs are not possible for a particular set of parameters or for a particular SRG. 
Four infinite families of genuinely nonabelian PDSs are described, two of which --  one arising from triangular graphs and one arising from Krein covers of complete graphs constructed by Godsil \cite{Godsil_1992} -- are new. We also include a new nonabelian PDS found by computer search and present some possible future directions of research.

\end{abstract}

\section{Motivation and Overview}
\label{sec:intro}

%
%

\noid
``Strongly regular graphs stand on the cusp between the random and the
highly structured."
(\cite{Cameron_2004})
			

\vs Research into strongly regular graphs and partial difference sets is a rich vein within algebraic combinatorics, involving graph theory, linear algebra, group theory, finite field theory, and algebraic number theory. (For definitions of a strongly regular graph and of a partial difference set, see Sections \ref{SRG} and \ref{PDS}, respectively.)  Constructions of PDSs correspond to projective two-weight codes \cite{Calderbank_Kantor_1986}, and PDSs also correspond to projective sets in projective geometries  \cite{Denniston_1969}.  For an excellent survey about early results related to PDSs and the connections PDSs have to various combinatorial objects, see \cite{Ma_1994b}.

Most of the work on PDSs has focused on abelian groups \cite{Ma_1984}, involving character theory, cyclotomy in finite fields, and computer searches for PDSs in high exponent groups \cite{Malandro_Smith_2020}.  We observe that many of the families of PDSs occur in $p$-groups, groups whose order is a power of a prime $p$, and the vast majority of $p$-groups are nonabelian. Just as one example, it is possible (using the DifSets package \cite{DifSets} in GAP \cite{GAP4}) to find all $(64,28,12)$ difference sets. There are $330159$ such difference sets giving $105269$ nonisomorphic symmetric designs. Of these, only $748$ designs have an abelian group acting sharply transitively on the points (this is the first time this observation has been made in print). This implies that only $0.7\%$ of $(64,28,12)$ symmetric designs arising from difference sets are constructible via a difference set in an abelian group. The conclusion we draw is that nonabelian groups are likely to contain vastly more constructions of interesting combinatorial objects, and we are only now starting to develop techniques to explore this nonabelian universe.

The purpose of this paper is to summarize the state of current knowledge regarding PDSs in nonabelian groups and how it relates to what is known in the abelian case, present some new results, and provide a number of possible future directions of interest.  The paper is organized as follows.  Section~\ref{SRG} introduces SRGs, gives a few important examples for the rest of the paper, and includes basic feasibility conditions for a SRG to exist. 
 Section~\ref{PDS} transitions to the main topic of this paper, namely PDSs and more generally automorphism groups of SRGs. Section~\ref{gen.nonabelian} defines ``genuinely nonabelian'' PDSs, and we provide a few examples of this concept. The first and last of these (Theorems \ref{mainconstruction} and \ref{thm:godsil}, respectively) are new infinite families of genuinely nonabelian PDSs, while the others are known examples that can be shown to have parameters that will not support an abelian PDS. Section~\ref{sect:nonab} sketches an outline of the techniques that have been used in exploring nonabelian PDSs, and we share some questions for further study.  Finally, Appendix~\ref{appendix} contains explicit details of a new genuinely nonabelian PDS with parameters $(512, 133, 24, 38)$.

%

\section{Introduction to Strongly Regular Graphs}
\label{SRG}
A graph is {\it regular} if there is an integer $k$ such that the degree of every vertex is $k$.
A regular graph with $v$ vertices and degree $k$ is {\it strongly regular} with parameters $(v,k,\lambda, \mu)$ if the number of paths of length two between two vertices $x$ and $y$ is dependent only on whether or not $x$ is adjacent to $y$.
If $x$ is adjacent to $y$ ($x \sim y$), then $\lambda$ represents the number of paths of length two from $x$ to $y$;
if $x$ is {\em not} adjacent to $y$ ($x \not\sim y$) then $\mu$ represents the number of paths of length two from $x$ to $y$.
Simple inclusion-exclusion shows that the complement of a $(v,k,\lambda,\mu)$ SRG is a SRG with parameters
$(v, v-k-1, v-2k+\mu-2, v-2k+\lambda)$

Strongly regular graphs were introduced by 
 R. C. Bose in 
  \cite{Bose_1963}.
 In that inaugural paper, Bose also introduced the concept of a partial geometry ${\rm pg}(r,k,t)$ and showed that the incidence relation on a partial geometry gave a SRG on the points of the geometry.
 Since that introduction, numerous examples of SRGs have been discovered \cite{Brouwer_VanMaldeghem_2022}, often expressed in terms of an underlying geometry  over a finite field.
 In particular, SRGs arise as the point graphs of partial geometries and generalized quadrangles.
 
 


The theory of strongly regular graphs is well developed.  See, for example,
 \cite[Chapter 21, 261--282]{vanLint_Wilson_2001},
 \cite[Sections 4.3--4.5, 82--100]{Biggs_White_1979}, 
  \cite[Chapter VII.11, 852--868]{Colbourn_Dinitz_2007},
 \cite[Ssections 9.4 \& 9.5, 118--125]{Hall_1986},
 \cite{Hubaut_1975}, and
 \cite[Section VIII.3, 270-276]{Bollobas_1998}.
The existence of a $(v,k,\lambda, \mu)$ strongly regular graph requires that the parameters satisfy a number of feasibility conditions.


\begin{lemma}[First feasibility condition for a SRG]
\label{counting condition}
 A requirement for the existence of a strongly regular graph with parameters $(v,k,\lambda, \mu)$ is 
$k(k-\lambda-1) = (v-k-1)\mu \label{feasibility1}$.
\end{lemma}


Strongly regular graphs for which $\mu=0$ have $k=\lambda+1$ and are disjoint copies of the complete graph $K_{k+1}$.  
If $\mu=k$ then Lemma~\ref{counting condition} implies that $v-k=k-\lambda$ and so $v-2k+\lambda=0.$
In this case the complementary graph has $\mu=0$ and so is the union of disjoint copies of a complete graph.
We consider these two cases as ``trivial."
From here on we will assume that $0 < \mu < k$.
Our assumption about $\mu$ also forces $\lambda < k-1$.
If a strongly regular graph has degree $k$ greater than or equal to $v/2$ then its complement is also a strongly regular graph with degree less than $v/2,$ so we will usually study graphs with $k < v/2.$

\noid
Every combinatorics research problem has built into it a basic question, ``Identify the objects in question and begin to list them..."  
We will begin with some examples.

%
%
%

\begin{example}
\label{triangular}
Consider a set $X$ of size $n$ and take as vertices the set $\begin{pmatrix} X \\ 2 \end{pmatrix}$, the subsets of all pairs of elements from $X$.  Two pairs are adjacent if they share exactly one element.  
This gives the triangular SRG $T_n$ with parameters \[\left(\dfrac{n(n-1)}{2}, 2(n-2), n-2, 4\right).\]
The complement of $T_n$ 
 is a SRG $T_n^C$ with parameters \[\left(\frac{n(n-1)}{2}, \frac{(n-2)(n-3)}{2},\frac{(n-4)(n-5)}{2}, \frac{(n-3)(n-4)}{2}\right).\]
 \end{example}

\begin{example}
\label{Paley}
Take a finite field $F$ of order $v=4\mu+1$. 
 Let $S$ be the set of nonzero squares in $F$ and join two vertices if their difference is a square, that is, $x\sim y \iff x-y \in S.$
 This graph turns out (with a little algebra) to be strongly regular with parameters $(4\mu+1, 2\mu, \mu-1,\mu).$
 Here the group $G=(F,+)$ acting sharply transitively on the graph is the additive group of the field.
This graph, known as the {\em Paley graph}~\cite{Paley_1933}, is isomorphic to its own complement.
\end{example}

\begin{example}
\label{LS}
There are two interesting infinite families of graphs that have the designation ``Latin Square."
Both families have $m^2$ vertices.
The (positive) latin square graphs  $PL_m(r)$ may be created by latin squares and have parameters 
$ v = m^2, k=r(m-1), \lambda = r^2-3r+m, \mu = r(r-1).$
This family is well known and graphs exist for most parameters.
The second infinite family of graphs are the negative latin square graphs, $NL_m(r)$;
these graphs have parameters $v = m^2, k=r(m+1), \lambda = r^2+3r-m, \mu = r(r+1).$
The existence of $NL_m(r)$ is rarer than in the first parameter set but graphs do arise in some cases in which a rank 3 group acts on an underlying field of order $m^2$ and may give rise to 2-weight codes.
  (See  \cite{vanLint_Schrijver_1981} for an early exploration of these graphs.)
  \end{example}

\subsection{The spectrum of a strongly regular graph}
%
The adjacency matrix $A$ of a graph is a $v \times v$ matrix whose rows and columns are indexed by the vertices, and the entry in the $(v_i, v_j)$ position is 1 if $v_i$ and $v_j$ are adjacent, $0$ otherwise. The adjacency matrix $A$ of a SRG satisfies the equation
\begin{equation*} A^2 = kI + \lambda A + \mu A^c, \end{equation*}
where $A^c = J-I-A$ is the adjacency matrix of the complementary graph.
(The matrix $J$ here is a square matrix of all ones.)
Substituting $J-I-A$ for $A^c$ and collecting like terms, we can rewrite this as
\begin{equation*} A^2 -(\lambda-\mu) A - (k-\mu) I  =\mu J. \end{equation*}
The $v\times 1$ column vector $\vec{1}$ is an eigenvector of $A$ with eigenvalue $k$.
If the graph $\Gamma$ is connected then the eigenspace corresponding to eigenvalue $k$ has dimension 1.
The eigenvectors with different eigenvalues will be orthogonal (since $A$ is symmetric)
and so any eigenvector with an eigenvalue distinct from $k$ will be be a root of the quadratic polynomial
\begin{equation*}F(x):=  x^2-(\lambda-\mu) x - (k-\mu). \label{quadratic}\end{equation*}
The discriminant of this quadratic polynomial is
\begin{equation*} \Delta := (\lambda-\mu)^2+4(k-\mu), \end{equation*}
and so the two roots of this polynomial are 
\begin{equation*} \theta_1, \theta_2 := \frac12 (\lambda-\mu \pm \sqrt{\Delta}). \end{equation*}
We will later need to think in terms of sum and difference of the nontrivial eigenvalues: $\theta_1+\theta_2 = \lambda-\mu;$ and 
$ \theta_1-\theta_2 = \sqrt{\Delta}.$ 
The difference between the nontrivial eigenvalues, $\sqrt{\Delta}$, will be an especially important parameter in the study of PDSs.

We summarize this information as a theorem.

\begin{thm}
\label{eigenvalues}
The eigenvalues of the adjacency matrix of a $(v,k,\lambda, \mu)$ SRG are $k, \theta_1, \theta_2$ for $F(\theta_1)=F(\theta_2)=0$.
\end{thm}

We can apply Theorem~\ref{eigenvalues} to the triangular graph found in Example~\ref{triangular}.

\begin{example}
A triangular graph $T_n$ has parameters 
\[\left(\dfrac{n(n-1)}{2}, 2(n-2), n-2, 4\right),\]
and so 
$\Delta = (n-2)^2,
 \theta_1+\theta_2 =n-6,$
and 
$  \theta_1-\theta_2 =n-2.$
This implies
$\theta_1=n-4, \theta_2=-2.$
\end{example}


\subsection{The Multiplicity Condition}
The eigenvalues of (the adjacency matrix of) a SRG lead us to two more feasibility conditions.
%
%
%
If we change the basis of our vector space $V$ from the standard basis to a basis of eigenvectors, then we essentially diagonalize the matrix $A$, with eigenvalues on the diagonal.
Let $m_i$ represent the number of times the eigenvalue $\theta_i$ occurs in this diagonal matrix, $i \in \{ 1,2 \}$; since $A$ is symmetric, all of the eigenvalues are real and hence we must have \begin{equation}
\label{traceI}
1+m_1+m_2 = v.                                                                                                                                                                                               
                                                                                                                      \end{equation}  The trace of a matrix is the sum of the elements on the diagonal, and the diagonal entries of $A$ are all zero; hence, $Tr(A) = 0,$ implying 
 \begin{equation} k + m_1\theta_1 + m_2\theta_2 = 0.  \label{traceA}  \end{equation}
Equations (\ref{traceI}) and (\ref{traceA}) place strong conditions on the positive integers $m_1$ and $m_2$ as shown when we apply these equations to the triangular graph.

\begin{example}
These two equations for a triangular graph $T_n$ give 
$m_1 + m_2 =v-1= (n^2-n-2)/2 = (n+1)(n-2)/2$
and
$m_1(n-4) + m_2(-2) =-2(n-2).$
Solving for $m_1$ and $m_2$ we have
%
$$m_1=n-1,\, m_2=\frac{n(n-3)}{2}.$$
\end{example}

%
%
%

\vs
Suppose that $\Delta =(\lambda-\mu)^2+4(k-\mu) $ is not the square of an integer.
Then $\sqrt{\Delta}$ is not rational and so the trace of the adjacency matrix is 
\[k + \frac12 m_1(\lambda-\mu+\sqrt{\Delta})+ \frac12 m_2(\lambda-\mu-\sqrt{\Delta})=0.\]
Thus, $(m_1-m_2)\sqrt{\Delta}$ must be rational. 
This can occur only if $m_1-m_2=0.$
This forces $m_1=m_2= (v-1)/2$ and our graph is a conference graph (defined to be a graph with the same parameters as a Paley graph, namely $(4\mu+1, 2\mu, \mu-1, \mu)$).
This gives our second feasibility condition:
\begin{lemma}
\label{conference graph}
If $\Gamma$ is a $(v,k,\lambda,\mu)$ SRG such that 
$\Delta = (\lambda-\mu)^2+4(k-\mu)$ is not a perfect square then $\Gamma$ is a conference graph.
\end{lemma}

This case is fairly restrictive.  The four parameters are stated exactly, in terms of the single parameter $\mu.$
Here the eigenvalues are 
\[\theta_i = \frac{-1+ (-1)^i \sqrt{4\mu+1}}{2}= \frac{-1+ (-1)^i \sqrt{v}}{2},\] and their multiplicities are the same: \[m_1 = m_2 = \frac{v-1}{2}.\]


%



A SRG where $\Delta$ is the square of an integer is called a Type II SRG.
The multiplicity of an eigenvalue 
\[\theta_1=\frac{(\lambda-\mu)+\sqrt{\Delta}}{2}\]
is
\[m_1=\frac{1}{2} \left((v-1)-\frac{2k+(v-1)(\lambda-\mu)}{\sqrt{\Delta}}\right).\]
This gives our third feasibility condition.
\begin{lemma}
\label{multiplicity}
If a $(v,k,\lambda,\mu)$ SRG satisfies
$\Delta = (\lambda-\mu)^2+4(k-\mu)$
 is a perfect square
then
\[m_1=\frac{1}{2} \left((v-1)-\frac{2k+(v-1)(\lambda-\mu)}{\sqrt{\Delta}}\right)\] is a nonnegative integer.
\end{lemma}

We will use Lemma~\ref{multiplicity} to rule out certain parameter sets from having an abelian PDS.

\section{Automorphisms of SRGs, Cayley Graphs, and PDSs}
\label{PDS}

  The history of SRGs is intimately connected with the study of the symmetries of such graphs.  A number of the sporadic simple groups (including the Higman-Sims group) were discovered as rank three permutation groups.  (A \textit{rank three permutation group} is a group acting transitively on a set such that the stabilizer of any element has exactly three orbits.)  Indeed, the constructions of simple groups led to an analysis of rank three permutation groups by Higman (\cite{Higman_1964},  \cite{Higman_1971a},  \cite{Higman_1971b}).
Rank three permutation groups give SRGs in a natural way: of the three orbits the stabilizer of a point has, one is the point itself, another corresponds to the neighbors of the point in the corresponding SRG, and the last orbit corresponds to the set of all non-neighbors.

 As with any combinatorial object, the symmetries (automorphisms) of SRGs are of interest. 
  The early constructions of SRGs, as geometric objects over a finite field, naturally gave rise to large automorphism groups.
  If a SRG has an automorphism group acting regularly (that is, sharply transitively) on vertices then we may construct the object as a \textit{Cayley graph}. Let $G$ be a finite group written multiplicatively and $S \subseteq G$ a subset with the property that $1 \not\in S$ and $S = S^{(-1)}$, that is, $S$ is closed under the property of taking inverses.  
We may then create a graph $\Cay(G,S)$ by agreeing that the vertex set is labeled by the set $G$, and two vertices $x, y \in G$ are adjacent (denoted $x \sim y$) if and only if $xy^{-1} \in S.$
The graph $\Cay(G,S)$ is the Cayley graph with group $G$ and generating set $S$.
  
The edge-defining property 
$xy^{-1} \in S$ can be expressed by saying that $x \sim y$ if and only if there is an element $s \in S$ such that $x=sy$.  
From this viewpoint, premultiplication by members of $S$ maps vertices (such as $y$) to adjacent vertices ($sy$).
Thus we can see that the set $S$ represents the neighborhood of the identity element and the size of the set $S$ is clearly the degree of the graph.
Thus a Cayley graph is regular.  
However, most Cayley graphs are not SRGs: the Paley graphs (squares in a finite field $F$ of characteristic $1 \bmod{4}$) and the Lattice graphs ($\{(x,0): x \ne 0\} \cup \{(0,x): x \ne 0\}$ in $\ZZ_n\times\ZZ_n$) are two Cayley graphs that are SRGs.

If $\Cay(G,S)$ is a Cayley graph, it is common to  apply adjectives describing  the group to the graph.  Thus we will speak of an {\em abelian} Cayley graph if $G$ is abelian, and so on.

Let $G$ be a finite group of order $v$ and $S$ a subset of $G$.
For $g \in G$, let $\lambda_g$ be the left regular representation. 
That is, order the elements of $G$ and construct the $v\times v$ $(0,1)$-matrix $\lambda_g$ where the $(i,j)$ entry of $\lambda_g$ is 1 if and only if $g(g_j)=g_i.$
Each matrix $\lambda_g$ is a permutation matrix describing how left-multiplication by $g$ permutes the elements of $G$.
If we define $\Lambda(S):= \sum_{s \in S} \lambda_s$, then $\Lambda(S)$ is a $(0,1)$-matrix with row and column sum equal to $|S|$.

Suppose we have a graph on $v$ vertices with adjacency matrix $A$.
Suppose, furthermore, that a group $G$ acts sharply transitively on the vertex set of the graph.
Choose a vertex $v_0$ and let $N(v_0)$ be the neighborhood of $v_0$.
Define
\[ S := \{g \in G: g(v_0) \in N(v_0)\}.\]
Then $S$ is a subset of $G$ with the property that $\Lambda(S)$ is an adjacency matrix for the graph.
If the underlying graph is strongly regular then $\Lambda(S)$, being an adjacency matrix for the underlying graph, has the same spectrum.

We now introduce a new perspective on a Cayley graph that is a SRG, which was introduced by S. L. Ma in  \cite{Ma_1984}. Let $G$ be a group and $X = \sum_{g \in G} a_g g$ an element of the group ring $\ZZ[G].$
Let $m$ be an integer. 
We define $X^{(m)} := \sum_{g \in G} a_g g^m$.
(Note the parentheses around the integer $m$, so that we distinguish $X^{(m)}$ from $X^{m}.$)
We will often follow the convention that a (sub)set $S$ of a group $G$ will correspond to the group ring element $S = \sum_{s \in S} s$, and the possible confusion between $S$ being a set and $S$ being a group ring element will be determined by the context. We note that the group ring element $S$ will satisfy $S^{(-1)} = S$, and that leads to the following definition.

\begin{defn}
\label{PDSdef}
A $(v,k,\lambda, \mu)$ {\it partial difference set} (PDS) in a group $G$ is a set $S$ such that, in the group ring $\ZZ[G]$,
\[ SG=kG, \ S^2=k+\lambda S+\mu(G-S-1).\]
\end{defn}

The discussion, above, gives the following result:
\begin{lemma}
\label{groupringPDS}
A $(v,k,\lambda, \mu)$  PDS $S$ is equivalent to a $(v,k,\lambda, \mu)$  Cayley graph $\Cay(G,S)$.
\end{lemma}

Indeed, this shows that a $(v, k, \lambda, \mu)$ PDS $S$, when considered as a subset of a group $G$, requires $|G| = v$, $|S| = k$, and every nonidentity element $g \in G$ can be written in either $\lambda$ or $\mu$ different ways (depending on whether or not $g$ is in $S$) as $xy^{-1}$, where $x, y \in S$.  Furthermore, we typically require $1 \notin S$ so that the corresponding SRG does not contain loops.  Moreover, we say that a PDS is a \textit{Type II PDS} if the corresponding SRG is a Type II SRG.

One tool that has been used in the abelian case to construct PDSs is known as {\em multipliers}. If $\phi$ is a group automorphism on a group $G$ and $S \subset G$ is a $(v,k,\lambda, \mu)$ PDS in $G$, then a simple computation shows that $\phi(S)$ is also a $(v,k,\lambda,\mu)$ PDS in $G$. If the automorphism takes the form $\phi(g) = g^m$ for an integer $m$, then we call the automorphism a {\em multiplier}. A powerful result in {\it abelian} PDSs is Ma's Multiplier Theorem \cite{Ma_1984}:

\begin{lemma}[Ma's Multiplier Theorem]
\label{multiplierthm}
If $S$ is a Type II PDS in an abelian group $G$ and if $m$ is an integer relatively prime to the order of $G$, then $S^{(m)} =S.$
\end{lemma}

We will observe that this lemma does not necessarily hold in nonabelian groups: namely, we will have explicit examples in the next section so that $S$ is a {\it nonabelian} PDS and $S^{(m)} \neq S$.

\section{Genuinely Nonabelian PDSs}
\label{gen.nonabelian}
We now come to the key definition in the paper. 
\begin{defn}
We will call a PDS {\it genuinely nonabelian} if the underlying strongly regular graph $\Cay(G,S)$ has no abelian automorphism group acting sharply transitively on the vertices.
We will call a set of parameters $(v,k,\lambda,\mu)$ {\it genuinely nonabelian} if there are nonabelian PDSs with those parameters but there are no abelian PDSs.
\end{defn}
The first known example of a genuinely nonabelian combinatorial object was the Hadamard difference set with parameters $(100, 45, 20)$ found in  \cite{Smith_1995}.  By removing the identity, we get a nonabelian PDS with parameters $(100,44, 18,20).$  
It had been proven previously that these objects did not exist in abelian groups.
Similarly, nonabelian PDSs found in  \cite{Jorgensen_Klin_2003} have parameters 
$(100,22,0,6),$ $(100,36,14,12),$ $(100,45,20,20),$ and $(100,44,18,20),$ and there are no abelian PDSs with these parameters. 

In order to show that a set of PDS parameters is genuinely nonabelian, we need to develop techniques for showing that no abelian PDS can exist with those parameters. The first technique requires that we define a character of an abelian group $G$: the map $\chi$ from $G$ to the multiplicative group $\C$ is called a character if $\chi$ is a homomorphism. The set of all characters of $G$ is denoted $G^*$. 
The following is well known.

\begin{lemma}
\label{dualgroup}
If $G$ is an abelian group, then $G^*$, under the operation of pointwise multiplication, is a group isomorphic to $G$.
\end{lemma}

Just as we can look for subsets of $G$ that satisfy the conditions of being a PDS, we can look for subsets of $G^*$ that are PDSs. The following result (due to Bridges and Mena~\cite{Bridges_Mena_1982} and Delsarte~\cite{Delsarte_1973}) identifies a subset of $G^*$ that will be a PDS.

\begin{thm}
\label{delsarte}
Let $S$ be a PDS in an abelian group $G$ and fix a nontrivial eigenvalue $\theta$ of $S$.
Then, $S^* = \{\chi \in G^*: \chi(S)=\theta\}$ is a PDS in $G^*$.
\end{thm}


The parameters of the PDS from Theorem \ref{delsarte} are a bit messy; we refer the interested reader to \cite[Theorem 3.4]{Ma_1994b} for full details. We do note, however, that $|S^*|$ in Theorem~\ref{delsarte} is the multiplicity of the eigenvalue $\theta$. We also know from Lemma~\ref{dualgroup} that $|G^*| = |G|$. Thus, we can consider the feasibility conditions from Section~\ref{SRG} to determine whether it is possible to have a PDS with those two parameters. If not, then it will be impossible to construct a PDS $S$ in the abelian group $G$. 
%
Another related result that can be used to rule out abelian PDSs in some parameters where a nonabelian PDS might still exist is due to Ma~\cite{Ma_1994a}.   

\begin{thm}
\label{discriminant nonexistence}
If $S$ is a PDS in an abelian group $G$ and \[S^*:= \left\{\chi \in G^*: \chi(S) = \frac{(\lambda-\mu)+\sqrt{\Delta}}{2}\right\},\] then $\sqrt{\Delta^*}= v/\sqrt{\Delta}.$
In particular, if $S$ is a Type II PDS in an abelian group then $v/\sqrt{\Delta}$ is an integer.
\end{thm}

We now apply these criteria for abelian PDSs to three particular sets of parameters. Later in this section, we will show that there are nonabelian PDSs with these parameters, providing new examples of genuinely nonabelian sets of parameters for PDSs.

\begin{cor}
\label{nonexistenceabeliantriangular}
If $n > 4$, then there is no abelian $(\frac{n(n-1)}{2}, 2(n-2), n-2, 4)$ PDS (these are the parameters for $T_n$).
\end{cor}

\bproof
The parameters of $T_n$ have $v=\frac{n(n-1)}{2}$ and $\Delta = (n-2)^2.$
If $n>4$ then 
\[ \frac{v}{\sqrt{\Delta}} = \frac{n(n-1)}{2(n-2)}=\frac12n+\frac12+\frac{2}{n-2}.\]
which is not an integer if $n>4.$  
\eproof

\begin{cor}
 \label{cor:godsil}
 Let $q$ be a prime power and $r < q+1$ be an integer dividing $q+1$.  A PDS with parameters
 \[\left(q^3, (q-1)\left( \frac{(q+1)^2}{r} - q\right), r \left( \frac{q+1}{r} - 1\right)^3 + r - 3, \left( \frac{q+1}{r} - 1\right)\left( \frac{(q+1)^2}{r} - q\right) \right)\]
 is genuinely nonabelian.
\end{cor}

\begin{proof}
 In this case, $\sqrt{\Delta} = q(q+1)/r$.  Since $(q+1)/r > 1$, $\sqrt{\Delta}$ does not divide $v = q^3$.  The result follows from Theorem \ref{discriminant nonexistence}.
\end{proof}

There are known to be SRGs with the parameters listed in Corollary \ref{cor:godsil} by the work of Godsil \cite[5.3 Lemma]{Godsil_1992}. We next highlight one specific instance of Corollary \ref{cor:godsil}; as we will see later in this sections, PDSs are known to exist for these parameters. 

\begin{cor}
\label{abelianGQ}
Let $S$ be a PDS with parameters $(q^3, q^2+q-2, q-2, q+2)$.  If $q$ is odd then $S$ cannot be an abelian PDS.
\end{cor}

\bproof
This follows from setting $r = (q+1)/2$ in Corollary \ref{cor:godsil}.
\eproof






As our first example of a genuinely nonabelian set of parameters, the following construction of a PDS in a nonabelian group occurs in a set of parameters that Corollary~\ref{nonexistenceabeliantriangular} shows cannot have an abelian PDS. This PDS construction is new.

\begin{thm}
\label{mainconstruction}
Let $q = p^r$ be a prime power congruent to $3$ mod $4$.
Then we may represent $T_q$ as a Cayley graph in the semidirect product $C_p^r \rtimes C_t$ where $t = (q-1)/2.$
\end{thm}

\begin{proof}
We prove this for $r=1$: the general case is similar.
Let $V$ be the set of pairs (sets of size 2) from $\ZZ_p$.
Define $\sigma : V \rightarrow V$ by $\sigma(\{a,b\}):= \{a+1,b+1\}$ where addition is done in $\ZZ_p$.
The function $\sigma$ is a permutation of $V$ of order $p$.

Let $g$ be a primitive root of $p$ and define $m:=g^2 \in \ZZ_p^*.$ 
The element $m$ generates a (multiplicative) subgroup,  the quadratic residues of $p$, of order $t$, index 2, in $\ZZ_p^*.$
Define $\tau: V \rightarrow V$ by $\tau(\{a,b\}):= \{ma,mb\}$ where multiplication is done in $\ZZ_p$.
The function $\tau$ is  a permutation of $V$ of order $t = (p-1)/2$.

We compute $\tau\sigma\tau^{-1}$:
\begin{align*} \tau\sigma\tau^{-1}(\{a,b\}) &= \tau\sigma(\{m^{-1}a,m^{-1}b\}) = \tau(\{m^{-1}a+1,m^{-1}b+1\})\\ 
 &= \{a+m,b+m\} = \sigma^m(\{a,b\}). 
\end{align*}

Therefore, 
\[ \tau\sigma\tau^{-1}=\sigma^m\]
Thus $G:=\langle \sigma, \tau \rangle \cong C_p\rtimes_m C_t$ is a nonabelian group of order $pt$ and all its elements may be described by $\sigma^i\tau^j$ as $i$ ranges through $\{0,1,2..,p-1\}$ and $j$ ranges through $\{0,1,2..,t-1\}.$

What is the orbit of $\{0,1\}$ under the group $G$? 
The powers of $\tau$ map $\{0,1\}$ to $\{0, jm\}$ for $j \in \{1,2,3...,t-1\}.$
The powers of $\sigma$ then map these elements to $\{a, a+jm\}$ for $a \in \{0,1,2, ..., p-1\}.$
The set 
\[S_+ := \left\{\{a, a+jm\}: a \in \{0,1,2, ..., p-1\}, j \in \{1,2,3...,t-1\} \right\}\]
is half of the vertices of $T_p$ and is a subset of the orbit of $\{0,1\}.$
Now  $\sigma^{-1}(\{0,1\})=\{0,-1\}$, and the orbit of $\{0,-1\}$ includes
\[ S_- := \left\{\{a, a-jm\}: a \in \{0,1,2, ..., p-1\}, j \in \{1,2,3...,t-1\} \right\}.\]
If $-1$ is a quadratic residue of $p$ then $-1 =mj$ for some integer $j$ and these two sets are the same.
But since $-1$ is {\em not} a quadratic residue we have that $S_+$ and $S_-$ are disjoint and $S_+ \cup S_- = T_p$.
In this case, $G=\langle \sigma,\tau \rangle$ is transitive on the vertices of $T_p$.
Since $G$ has the same order as $T_p$ and is transitive on $T_p$ then $G$ acts regularly on the graph $T_p.$
Thus we have constructed a PDS in $G$.
\end{proof}

\begin{rem}
We can explicitly write out a PDS when $G \cong C_p \rtimes_m C_{\frac{p-1}{2}}$.
Set 
\[ N:= \{ \{0,a\}: a=2,...,p-1\} \cup  \{ \{1,b\}: b=2,...,p-1\},\]
the neighborhood of $\{0,1\}$, and define
\[ S := \{ g \in G: g(\{0,1\})\in N\}.\]
We observe that since $\{1,2\} \in N$ and $\{0,-1\}\in N$ then $\{\sigma, \sigma^{-1}\}\subseteq S.$
Set $T := \{\tau^j: j \in \{1,2,...,p-1\}\}.$
Since $\{0,mj\} \in N$ for all $j$ then $T \subseteq S.$
Similarly, since $\{0,-mj\} \in N$ for all $j$ then $T\sigma^{-1} \subseteq S.$
The remaining elements of $S$ can be found by considering elements of the form $\{1,1+mj\}$ and $\{1,1-mj\}$ as $mj$ varies across the quadratic residues of $p$.
As the sets $\{1+mj:  j \in \{1,2,...,p-1\}\}$ and $\{1-mj:  j \in \{1,2,...,p-1\}\}$ partition the nonzero elements of $\ZZ_p$ and are a subset of $N$ then $\sigma T, \sigma T \sigma^{-1} \subseteq S.$
Thus,
\[ S = \{\sigma, \sigma^{-1}\} \cup  T \cup T\sigma^{-1} \cup \sigma T\cup \sigma T\sigma^{-1}.\]
(Note that the sets $T$ and $\sigma T \sigma^{-1}$ must be disjoint, and this is possible {\em only} if $G$ is nonabelian!)
If we wish to write all the elements of $G$ in the form $\sigma^i\tau^j$ then 
\[ S= \{\sigma, \sigma^{-1}\} \cup  T \cup \sigma^{-m}T \cup \sigma T\cup \sigma^{-m+1} T .\]

We further note that the requirement that $q$ is $3 \bmod{4}$ is necessary in Theorem~\ref{mainconstruction}. For example, suppose $p=5$.  Here $\tau\sigma\tau^{-1} = \sigma^4$ and $m=4$. Then
\[ \sigma= (01\ 12\ 23\ 34\ 04)(02\ 13\ 24\ 03\ 14) \text{ and } \tau = (01\ 04)(02\ 03)(12\ 34)(13\ 24)(14)(23).\]
The orbit of 01 is $\{01, 12, 23, 34, 04\}$ and the orbit of 02 is $\{02, 13, 24, 03, 14\}$.
In this case there is no PDS since these permutations are not transitive on the graph.

We also note that exploration in GAP \cite{GAP4} using the Algebraic Graph Theory package \cite{AGT} shows that $T_n$ is {\em not} a Cayley graph if $n \in \{5,6,8,9,10,12,13,14,15\}.$  
(The integer 15 required considerable work.)
One might suspect that the only time $T_n$ has a regular automorphism group is when $p$ is a prime power congruent to $3$ mod $4$.
\end{rem}


\begin{example}
As an example of Theorem~\ref{mainconstruction}, we consider $T_7$.
The permutation $\sigma$ in the proof has cycle structure
\[ \sigma= (01\ 12\ 23\ 34\ 45\ 56\ 06)(02\ 13\ 24\ 35\ 46\ 05\ 16)(03\ 14\ 25\ 36\ 04\ 15\ 26),\]
and the permutation $\tau$ has cycle structure
\[ \tau = (01\ 02\ 04)(03\ 06\ 05)(12\ 24\ 14)(13\ 26\ 45)(15\ 23\ 46)(16\ 25\ 34)(35\ 36\ 56).\]
In this case, a $(21,10,5,4)$ PDS in $\langle \sigma, \tau \rangle \cong C_7 \rtimes_2 C_3$ is 
\[ S = \{\sigma, \sigma^6, \ \tau, \tau^2,     \ \sigma\tau,  \sigma\tau^2,  \ \tau\sigma^6, \tau^2\sigma^6, \ \sigma\tau\sigma^6, \sigma\tau^2\sigma^6\}\]
\[= \{\sigma, \sigma^6,  \tau, \tau^2,  \sigma \tau, \sigma \tau^2, \sigma^5 \tau, \sigma^3 \tau^2, \sigma^6\tau, \sigma^4\tau^2\}.\]

We note that this PDS does not satisfy Ma's Multiplier Theorem (Lemma \ref{multiplierthm}): while $2$ is relatively prime to $|\langle \sigma, \tau \rangle| = 21$, $\sigma \in S$ whereas $\sigma^2 \notin S$. 
\end{example}

A second example of genuinely nonabelian PDSs have parameters $(q^3, q^2+q-2, q-2, q+2)$. We showed in Corollary~\ref{abelianGQ} that no abelian PDSs can have these parameters if $q$ is odd. Two separate constructions (\cite{Kantor_1986}, \cite{Swartz_2015}) provide nonabelian PDSs with these parameters, the first when $q$ is an odd prime power and the associated group is a Heisenberg group over the field of order $q$ (that is, the set of unipotent upper-triangular matrices with entries in ${\rm GF}(q)$), and the second when $q$ is an odd prime and the associated group is an extraspecial group of exponent $q^2$. The parameter set is hence genuinely nonabelian, and any SRG coming from these PDSs will also be genuinely nonabelian. 

The smallest example of an odd $q$ is the SRG with parameters $(27,10,1,5)$ created by setting $q=3.$ Consider the nonabelian Heisenberg group 
%
%
\[ G_1 := (C_3 \times C_3) \rtimes C_3 := \langle x, y, z : x^3 = y^3 = z^3 = xzx^{-1} = yzy^{-1} = 1, xyx^{-1} = yz^{-1} \rangle.\]  
In this nonabelian group,
\[ S_1 = (x + x^2 + xy + xy^2) + (1 +  y + x^2y^2)z + (1 + y^2 +  x^2y)z^2\]
is a $(27,10, 1, 5)$ PDS. 

Next, consider the nonabelian group
\[ G_2 :=C_9 \rtimes_4 C_3 := \langle x,y : x^9=y^3=1, yxy^{-1}=x^4\rangle.\]
In this nonabelian group,
\[S_2:= (x^2+x^3+x^6+x^7)+(x^2+x^3+x^4)y+(x+x^2+x^6)y^2\]
is a $(27,10,1,5)$ PDS.  In particular, the PDS $S_2$ in the group $G_2$ fails Ma's Multiplier Theorem (Lemma \ref{multiplierthm}): the integer $2$ is relatively prime to the order of $G_2$ and $x^2 \in S_2$, but $(x^2)^2 = x^4 \notin S_2$.
  
  While the groups $G_1$ and $G_2$ are not isomorphic, the corresponding SRGs $\Cay(G_1, S_1)$ and $\Cay(G_2, S_2)$ are in fact isomorphic: they are each the complement of the Schl\"afli graph, which was proven to be unique by Seidel \cite{Seidel_1968}.  (On the other hand, the graphs arising from extraspecial groups of order $p^3$ with exponent $p$ and exponent $p^2$ are not isomorphic in general; for example, they are not isomorphic when $p = 5$.)

These PDSs happen to be \textit{pseudo-geometric}: they have the same parameters as a SRG that would be a collinearity graph of a \textit{generalized quadrangle} (GQ).  A GQ of order $(s,t)$ is a point-line incidence geometry such that any two points lie on at most one common line; every line is incident with exactly $s+1$ points; every point is incident with exactly $t+1$ lines; and, given a point $P$ and a line $\ell$ not incident with $P$, there is a unique point on $\ell$ collinear with $P$.  Indeed, a group acting regularly on the set of points of a GQ of order $(s,t)$ would correspond to a $((s+1)(st+1), s(t+1), s-1, t+1)$ PDS.  We refer the interested reader to \cite{Payne_Thas_2009} for more information about GQs.

Groups acting regularly on the set of points of a GQ have produced extremely interesting examples of nonabelian PDSs in recent years.  For example, Bamberg and Giudici noted the existence of a $(4617, 520, 7, 65)$ PDS in a group isomorphic to $C_{513} \rtimes C_9$ in \cite{Bamberg_Giudici_2011} that is genuinely nonabelian.  Furthermore, Bamberg and Giudici noted that other groups (beyond just Heisenberg groups) can act regularly on the set of points of a Payne-derived GQ of ${\rm W}(q)$ of order $(q-1, q+1)$ for $q$ odd (which would correspond to a $(q^3, q^2+q-2, q-2, q+2)$ PDS).  In fact, Feng and Li \cite{Feng_Li_2021} classified all groups acting regularly on the Payne-derived GQ of ${\rm W}(q)$, $q$ odd, showing as a byproduct that such groups can actually have \textit{unbounded nilpotency class}.  

We end this section by proving that, for every choice of $r < q + 1$ in Corollary \ref{cor:godsil}, there is a genuinely nonabelian PDS with those parameters.  As far as we know, other than when $r = (q+1)/2$, these have generally not previously been recognized as feasible parameters for PDSs.

\begin{thm}
 \label{thm:godsil}
 Let $q$ be a prime power and $r < q+1$ be an integer dividing $q+1$.  There exists a genuinely nonabelian PDS with parameters
 \[\left(q^3, (q-1)\left( \frac{(q+1)^2}{r} - q\right), r \left( \frac{q+1}{r} - 1\right)^3 + r - 3, \left( \frac{q+1}{r} - 1\right)\left( \frac{(q+1)^2}{r} - q\right) \right).\]
\end{thm}

\begin{proof}
 As noted above, SRGs with these parameters do exist: for example, by \cite[5.3 Lemma]{Godsil_1992}, since $q$ is a prime power and $r$ divides $q + 1$, such a graph can be constructed from a GQ of order $(q^2, q)$ with a particularly nice \textit{ovoid}, that is, a collection of $q^3 + 1$ points in the GQ that are pairwise noncollinear.  We again refer the reader to \cite{Payne_Thas_2009} for information and terminology related to GQs.  Start with the classical GQ ${\rm H}(3,q^2)$ of order $(q^2, q)$, and find an automorphism $g$ of order $(q+1)/r$ fixing each point in an ovoid (but fixing no line of the GQ). Take one point $P_0$ in the ovoid, and fix one distinguished line $\ell_0$ incident with $P_0$.  The vertices of our SRG will be the $\langle g \rangle$-orbits of lines ${\langle g \rangle}\ell$ such that $\ell$ is not incident with $P_0$ but $\ell_0$ is concurrent with some line in $\langle g \rangle \ell$, and two orbits $\langle g \rangle \ell_1$ and $\langle g \rangle \ell_2$ are adjacent in the SRG if and only $\ell_1$ is collinear with some line in $\langle g \rangle \ell_2$.  (We note that the roles of points and lines are often reversed in this construction, and so typically the construction is done in the dual GQ ${\rm Q}^-(5,q)$ of order $(q, q^2)$; we have chosen to work instead in ${\rm H}(3,q^2)$ for the nice representation we have of an ovoid.)     
 
 We now explicitly construct the GQ and ovoid.  This construction is essentially the same as that of \cite[6.1 Lemma]{Godsil_1992}, except we are choosing a different nonsingular Hermitian form so that our group $G$ containing the PDS has a nice representation.  Let $V = \GF(q^2)^4$, and, given two vectors $x, y \in V$, define a nonsingular Hermitian form $b: V \times V \to \GF(q^2)$ on $V$ by
 \[ b(x, y) := x_1 y_1^q + x_2 y_4^q + x_3 y_3^q + x_4 y_2^q.\]  The $1$-dimensional totally isotropic subspaces of $V$ -- that is, the $1$-dimensional subspaces $\langle x \rangle$ spanned by the (nonzero) vectors $x$ such that $b(x,x) = 0$ -- form the point set of a GQ $\calH$ of order $(q^2, q)$ in $\PG(3,q^2)$.  The lines of $\calH$ are the $2$-dimensional totally isotropic subspaces -- that is, those $2$-dimensional subspaces $U$ such that $b(x,y) = 0$ for all $x,y \in U$ -- with incidence between points and lines defined by subspace containment.  
 
 To construct an ovoid, we note that the intersection of $\calH$ with any non-tangent hyperplane is an ovoid; see \cite{Godsil_1992} and \cite{Payne_Thas_2009}.  For our choice of Hermitian form, we may choose $x_1 = 0$ as our non-tangent hyperplane. Thus, our ovoid is 
 \[ \calO := \left\{ \langle x \rangle : x_1 = 0, \, x_2 x_4^q + x_3^{q+1} + x_4 x_2^q = 0\right\}.\]
 Moreover, by \cite[p. 249]{Dixon_Mortimer_1996}, we see in fact that
 \[ \calO = \left\{ \langle (0,1,0,0)^t \rangle \right\} \cup \left\{ \langle (0,\alpha, \beta, 1)^t \rangle: \alpha + \alpha^q + \beta^{q+1} = 0, \alpha, \beta \in \GF(q^2)\right\},\]
and $|\calO| = q^3 + 1$.

The element $g$ of order $(q+1)/r$ fixing each point in $\calO$ but fixing no line of $\calH$ is constructed exactly as in \cite[6.1 Lemma]{Godsil_1992}: if $\gamma$ is an element of $\GF(q^2)$ with multiplicative order $(q+1)/r$, then we define
\[ g := \begin{pmatrix}
         \gamma & 0 & 0 & 0\\
         0      & 1 & 0 & 0\\
         0      & 0 & 1 & 0\\
         0      & 0 & 0 & 1\\
        \end{pmatrix}.
\]

Since $\gamma^{q+1} = 1$, left multiplication by $g$ preserves the bilinear form $b$.  It is also clear that $g$ fixes every element of $\calO$, and, for the same reasons as noted in \cite[6.1 Lemma]{Godsil_1992}, $\langle g \rangle$ acts fixed-point freely on the points of $\calH$ not on the hyperplane $x_1 = 0$, and so $\langle g \rangle$ acts fixed-point freely on the lines of $\calH$ incident with each point in $\calO$.

We now construct a subgroup $G$ of automorphisms of $\calH$ of order $q^3$.  Let $\alpha, \beta \in \GF(q^2)$ be elements satisfying $\alpha + \alpha^q + \beta^{q+1} = 0$; there are exactly $q^3$ such pairs $(\alpha, \beta)$.  Define
\[ t_{\alpha, \beta} := \begin{pmatrix}
                         1 & 0 &        0 &      0\\
                         0 & 1 & -\beta^q & \alpha\\
                         0 & 0 &        1 &  \beta\\
                         0 & 0 &        0 &      1\\
                        \end{pmatrix}.
\]
For the same reasons as noted in \cite[p. 249]{Dixon_Mortimer_1996}, the set
\[ G:= \left\{ t_{\alpha, \beta} : \alpha + \alpha^q + \beta^{q+1} = 0\right\}\]
is a group of order $q^3$ fixing the point $\langle (0,1,0,0)^t \rangle$ of $\calO$ and acting regularly on the remaining $q^3$ elements of $\calO$.

Finally, we claim that the group $G$ acts regularly on the vertices of our associated SRG.  To see this, we choose $P_0 = \langle (0,1,0,0)^t \rangle$.  We note that, for all $\alpha, \beta$, $g t_{\alpha, \beta} = t_{\alpha, \beta} g$, so the group $G$ preserves the $\langle g \rangle$-orbits of lines.  Since the number $r$ of $\langle g \rangle$-orbits of lines incident with $P_0$ is coprime to $|G| = q^3$, at least one $\langle g \rangle$-orbit of lines incident with $P_0$ is fixed by $G$; we choose one such distinguished $\langle g \rangle$-orbit $\langle g \rangle \ell_0$.  Since $G$ fixes $\langle g \rangle \ell_0$ but acts regularly on $\calO \backslash \{ P_0 \}$, $G$ acts regularly on the $\langle g \rangle$-orbits of lines $\ell$ such that $\ell$ is not incident with $P_0$ but $\ell_0$ is concurrent with some line in $\langle g \rangle \ell$; in other words, $G$ acts regularly on the vertices of the associated SRG.  The group $G$ is clearly nonabelian from its definition, but we may also conclude that $G$ is nonabelian from Corollary \ref{cor:godsil}.  This completes the proof. 
\end{proof}

\begin{rem}
\label{rem:godsil}
 For each fixed prime power $q = p^d$, each PDS constructed in Theorem \ref{thm:godsil} occurs in the \textit{same} group $G$ of order $q^3$; indeed, the construction in Theorem \ref{thm:godsil} shows that $G$ is isomorphic to a Sylow $p$-subgroup of the projective special unitary group ${\rm PSU}(3,q)$.  Moreover, examination of the elements $t_{\alpha, \beta}$ shows that $G$ has exponent $4$ when $q$ is even and exponent $p$ when $q$ is odd.  For example, although it is already apparent from the matrix structure of the elements $t_{\alpha, \beta}$, when $q = p$ is an odd prime, this shows that $G$ must be a Heisenberg group of order $p^3$.
\end{rem}

As we will discuss further in the next section, the results in this section suggest that the examination of known combinatorial structures for sharply transitive subgroups may be a fruitful line of inquiry.

\section{Nonabelian Techniques and Future Directions}
\label{sect:nonab}

The world of PDSs in nonabelian groups is relatively unexplored territory, especially when compared to the abelian case.  The main issue at the moment is that powerful techniques in the abelian setting -- such as Ma's Multiplier Theorem (Lemma \ref{multiplierthm}) -- simply do not hold when the group is nonabelian.  With respect to the Multiplier Theorem, there is a partial generalization to the nonabelian case due to Ma \cite{Ma_1984}:  

\begin{lemma}
\label{lem:multnonab}
 Let $S$ be a Type II PDS in a group $G$, and let $m$ be an integer that is relatively prime to $|G/G'|$, where $G'$ is the derived subgroup of $G$.  If $\overline{S}$ denotes the image of $S$ under the natural homomorphism $G \to G/G'$, then $\overline{S}^{(m)} = \overline{S}$.
\end{lemma}

Unfortunately, this result is not always useful in practice: for example, we consider again the group $G = \langle \sigma, \tau \rangle \cong C_7 \rtimes_2 C_3$ and PDS $S$ arising from Theorem \ref{mainconstruction}.  As we saw in Section \ref{gen.nonabelian}, $m = 2$ is not a multiplier for $S$, but Lemma \ref{lem:multnonab} tells us that $m = 2$ will be a multiplier for $\overline{S}$ in $G/G'$.  However, in this case, $G' = \langle \sigma \rangle$ and $G/G' \cong C_3$, and so $\overline{S} = \overline{S}^{(2)} = \overline{S}^{(-1)}$, which we already know will be true since $S = S^{(-1)}$ for a PDS. 

\subsection{A useful class function}

There are a few known results that are extremely useful in the nonabelian setting, and these results have largely been inspired by the study of groups acting regularly on the point set of a generalized quadrangle.  The first result is a generalization of Benson's Lemma \cite[Lemma 4.3]{Benson_1970} to SRGs.  Benson's Lemma utilizes a technique attributed to Graham Higman, which calculates the value of a character of the automorphism group of an association scheme on an eigenspace; see \cite[pp. 89--91]{Cameron_1999}.  De Winter, Kamischke, and Wang \cite{DeWinter_etal_2016} proved a generalization for SRGs, which we state here specifically for PDSs; recall that $\sqrt{\Delta} = \theta_1 - \theta_2$.

\begin{lemma}
 Let $S$ be a Type II PDS with parameters $(v,k,\lambda, \mu)$ with eigenvalues $\theta_2 < \theta_1 < k$ in a group $G$, and let $x$ be a nontrivial element of the group $G$.  If $d_1(x)$ denotes the number of vertices of the corresponding SRG $\Gamma$ that $x$ sends to adjacent vertices in $\Gamma$, then
 \[ k - \theta_2 \equiv \mu - \theta_2(\theta_1 + 1) \equiv d_1(x) \pmod {\sqrt{\Delta}}.\]
\end{lemma}

Yoshiara \cite{Yoshiara_2007} used Benson's Lemma along with clever counting and group theoretical arguments to prove that no GQ of order $(t^2, t)$, where $t \geqslant 2$, has a group of automorphisms acting regularly on its point set.  Inspired by these results, Swartz and Tauscheck \cite{Swartz_Tauscheck_2021} proved corresponding results that apply to PDSs.  We state an equivalent result here based on the notation of this paper; in what follows, $\Cl(x)$ denotes the conjugacy class of the group element $x \in G$ and $C_G(x)$ denotes the centralizer of $x$ in $G$.  

\begin{lemma}
\label{lem:SwartzTauscheck}
 Let $S$ be a Type II PDS in a group $G$, let $x$ be a nontrivial element of $G$, and define  
 \[ \Phi(x) := |\Cl(x) \cap S||C_G(x)|.\]
 Then,
 \[ \Phi(x) \equiv \mu - \theta_2(\theta_1 + 1) \pmod {\sqrt{\Delta}}.\]
 In particular, if $\sqrt{\Delta}$ does not divide $\mu - \theta_2(\theta_1 + 1)$, then $\Cl(x) \cap S \neq \varnothing$.
\end{lemma}

Note that $\Phi(x)$ is invariant on each conjugacy class of $G$ and hence a class function.  In particular, \textit{every} nonidentity central element of a group is its own conjugacy class.  Since the complement of a SRG is also a SRG, we can apply to the criterion of Lemma \ref{lem:SwartzTauscheck} to both a PDS and its complement in $G\backslash \{1\}$ to prove the following:

\begin{cor}
\label{cor:SwartzTauscheck}
 \begin{itemize}
  \item[(i)] If $\sqrt{\Delta}$ divides neither $\mu - \theta_2(\theta_1 + 1)$ nor $v - 2k + \lambda - \theta_2(\theta_1 + 1)$, then a group with a nontrivial center cannot contain a $(v, k, \lambda, \mu)$ PDS.
  \item[(ii)] If $\sqrt{\Delta}$ does not divide one of $\mu - \theta_2(\theta_1 + 1)$ or $v - 2k + \lambda - \theta_2(\theta_1 + 1)$, then a $(v, k, \lambda, \mu)$ PDS is genuinely nonabelian.
 \end{itemize}
\end{cor}

\begin{proof}[Proof of (ii)]
 We include a proof of (ii), since it is not stated explicitly in \cite{Swartz_Tauscheck_2021}.  We assume without loss of generality that $\sqrt{\Delta}$ does not divide $\mu - \theta_2(\theta_1 + 1)$.  By Lemma \ref{lem:SwartzTauscheck}, for any nonidentity element $x$ in a group $G$ of order $v$, $\Cl(x) \cap S \neq \varnothing$.  In particular, $G$ cannot be abelian, since this implies every nonidentity element is contained in the PDS, and so any such PDS is necessarily genuinely nonabelian.  
\end{proof}

We include two examples illustrating the utility of these ideas.

\begin{example}
 Consider a $(p^3, p^2 + p - 2, p - 2, p + 2)$ PDS, where $p$ is an odd prime.  In this case, $\theta_1 = p - 2$, $\theta_2 = -p - 2$, $\sqrt{\Delta} = 2p$, $\mu - \theta_2(\theta_1 + 1) = p(p + 2)$, and $v - 2k + \lambda - \theta_2(\theta_1 + 1) = p^2(p - 1)$.  Thus, $\sqrt{\Delta}$ divides $v - 2k + \lambda - \theta_2(\theta_1 + 1)$ but not $\mu - \theta_2(\theta_1 + 1)$, and hence any PDS is genuinely nonabelian by Corollary \ref{cor:SwartzTauscheck}.  
 
 Let $G$ be an extraspecial group of order $p^3$ with exponent $p^2$ isomorphic to $C_{p^2} \rtimes C_p$. The centralizer of any noncentral element has order $p^2$, and hence there are
\[ \frac{p^3 - p}{p} + (p - 1) = p^2 + p - 2\]
nonidentity conjugacy classes in $G$.  Since every nonidentity conjugacy class must meet a $(p^3, p^2 + p - 2, p - 2, p + 2)$ PDS nontrivially by Lemma \ref{lem:SwartzTauscheck}, this means, if such a group contains a $(p^3, p^2 + p - 2, p - 2, p + 2)$ PDS, then every nonidentity conjugacy class would necessarily meet the PDS in exactly one element.  This is exactly what happens in the PDS constructed in \cite{Swartz_2015}; see, in particular, \cite[Lemma 5]{Swartz_2015}.
\end{example}

\begin{example}
For a triangular graph $T_n$ with parameters
$(n(n-1)/2, 2(n-2), n-2, 4)$, $\theta_1= n-4, \theta_2 = -2$ and $\sqrt{\Delta}=n-2.$
Thus, 
\[\mu-\theta_2(\theta_1+1)=4+2(n-3)=2n-2 \equiv 2 \pmod {\sqrt{\Delta}},\]
so by Lemma \ref{lem:SwartzTauscheck} every nonidentity conjugacy class meets the PDS nontrivially.

The conjugacy classes of $C_7 \rtimes C_3$ are $\{1\}$, $\Cl(\sigma),$ $\Cl(\sigma^3)$, $\Cl(\tau)$, and $\Cl(\tau^2)$ of sizes $1$, $3$, $3$, $7$, and $7$, respectively.
So, for $x\ne 1$,  $\Phi(x)$ has values 
\[ 7|\Cl(\sigma) \cap S|, \; 7|\Cl(\sigma^3) \cap S|, \; 3|\Cl(\tau) \cap S|,\; 3|\Cl(\tau^2) \cap S|.\]
Since $\sqrt{\Delta} = 5$, these must be each must be congruent to 2 modulo 5 while adding up to $k=10.$
This forces 
\[ |\Cl(\sigma) \cap S| = |\Cl(\sigma^3) \cap S|=1,\; |\Cl(\tau) \cap S|=|\Cl(\tau^2) \cap S| = 4.\]
With this information, it is easy to write out the PDS.
\end{example}

The constraints imposed by the function $\Phi$ appear to be extremely powerful in certain contexts.  For example, Ott \cite{Ott_2023} recently used similar ideas to prove that any group $G$ containing a $((s+1)(s^2+1), s^2+s+1, s+1)$ reversible difference set containing the identity must have even order, settling a thirty-year-old conjecture of Ghinelli \cite{Ghinelli_1992}; this equivalently shows that any group $G$ containing a $((s+1)(s^2+1), s^2+s, s-1, s+1)$ PDS must have even order. 

\begin{question}
 Can further results about the function $\Phi$ be used to prove the nonexistence of other PDSs in nonabelian groups?
\end{question}

Obviously, character theory has been an extremely powerful tool in the analysis of PDSs in abelian groups.  Moreover, the function $\Phi$, as noted above, is a class function.  We may replace the family of characters of an abelian group with the set of irreducible representations of a general group.  This worked in finding a $(100, 44, 18,20)$ genuinely nonabelian PDS; see  \cite{Smith_1995}.  Unfortunately, analysis of the irreducible representations of degree greater than one requires a study of matrices over the complex numbers, a much more complicated theory.

\begin{question}
 How can we use the irreducible representations of a nonabelian group $G$ to study the PDSs contained in $G$?
\end{question}

\subsection{New PDSs from known combinatorial objects}

As demonstrated by the tremendous success of Feng and Li finding new examples of PDSs by studying the full automorphism groups of known generalized quadrangles \cite{Feng_Li_2021} -- as well as the results of Theorems \ref{mainconstruction} and \ref{thm:godsil} -- it stands to reason that the study of other combinatorial objects with large automorphism groups could prove fruitful.  Moreover, J\o rgensen and Klin constructed $15$ distinct PDSs (with parameters $(100,22,0,6)$, $(100,36,14,12)$, $(100,45,20,20)$, $(100,44,18,20)$) in four distinct nonabelian groups of order $100$ by proceeding exactly in this manner \cite{Jorgensen_Klin_2003}.  Indeed, following along this line of reasoning, Feng, He, and Chen constructed nonabelian $2$-groups with exponent $4$, $8$, and $16$ and of nilpotency class $2$, $3$, $4$, and $6$ for the Davis-Xiang graphs and RT2 graphs \cite{Feng_He_Chen_2020} (which are Latin Square type; see Example \ref{LS}), and they remark that their ``results suggest that a good way to construct nonabelian groups that contain nontrivial partial difference sets and amorphic nonabelian Cayley association schemes is to consider the regular subgroups of the known strongly regular graphs and known amorphic association schemes.''  

We will further show the utility of this approach by constructing examples of $(512, 133, 24, 38)$ PDSs.  (Such a PDS arises from Theorem \ref{thm:godsil} with $q = 8$ and $r = 3$.)  In this case, $\theta_1 = 5$, $\theta_2 = -19$, $\sqrt{\Delta} = 24$, and $\mu - \theta_2(\theta_1 + 1) = 152$, so, by Lemma \ref{lem:SwartzTauscheck}, every nontrivial conjugacy class must meet such a PDS, which we already know will be genuinely nonabelian by Corollary \ref{cor:godsil}.  

Proceeding as in the proof of Theorem \ref{thm:godsil}, we may construct a SRG with these parameters by starting with a GQ ${\rm H}(3,64)$ of order $(64, 8)$, and, indeed, this graph can be constructed using the packages FinInG \cite{FinInG} and GRAPE \cite{GRAPE} in GAP \cite{GAP4}.  The full automorphism group of the graph has order $193536$, and a Sylow $2$-subgroup of the full automorphism group has order $1024$.  While Theorem \ref{thm:godsil} guarantees the existence of one group of order $512$ acting regularly on the vertices of this graph, an examination of all maximal subgroups of a Sylow $2$-subgroup yields the following:

\begin{thm}
 \label{thm:512}
 There exist two nonisomorphic nonabelian groups of order $512$ that contain $(512, 133, 24, 38)$ PDSs.
\end{thm}

\begin{proof}
 One group is the one constructed in Theorem \ref{thm:godsil}, and the other group and corresponding PDS are listed explicitly in Appendix \ref{appendix}.  They are nonisomorphic since the group constructed in Theorem \ref{thm:godsil} has exponent $4$ (see Remark \ref{rem:godsil}), while the group listed in Appendix \ref{appendix} (identified as \verb|SmallGroup(512,4508)| in GAP \cite{GAP4}) has exponent $8$.
\end{proof}

In light of Theorem \ref{thm:512}, the authors believe that a full examination -- in the spirit of \cite{Feng_Li_2021} -- of the groups acting regularly on the vertices of the Krein covers of complete graphs constructed by Godsil \cite{Godsil_1992} would be quite interesting.  With this in mind, we end the paper with one final question.

\begin{question}
 Which known SRGs have a nonabelian group of automorphisms acting regularly on vertices?  In particular, are there other infinite families of such SRGs and PDSs?
\end{question}

\newcommand{\etalchar}[1]{$^{#1}$}

\appendix
\section{Appendix: The second (512, 133, 24, 38) PDS}
\label{appendix}

In this Appendix, we include explicit code that can be read into GAP \cite{GAP4} for the second $(512, 133, 24, 38)$ PDS mentioned in Theorem \ref{thm:512}.

\begin{verbatim}
G2:= SmallGroup(512, 4508);;

gen:= GeneratorsOfGroup(G2);;

f1:= gen[1];;
f2:= gen[2];;
f3:= gen[3];;
f4:= gen[4];;
f5:= gen[5];;
f6:= gen[6];;
f7:= gen[7];;
f8:= gen[8];;
f9:= gen[9];;

S2:= [ f9, f7, f7*f9, f7*f8, f7*f8*f9, f8*f9, f8, f4*f5*f6*f7*f9, 
  f4*f5*f6*f7, f1*f2*f3*f4*f6*f7*f8, f1*f2*f3*f5*f7*f9, f1*f2*f3*f5, 
  f1*f2*f3*f4*f6*f8*f9, f1*f2*f3*f5*f6*f7*f9, f1*f2*f3*f5*f6*f7, 
  f1*f2*f3*f4*f8, f1*f2*f3*f4*f8*f9, f6*f8, f4*f5*f8, f4*f5*f7, 
  f6*f7*f9, f1*f2*f3*f6*f9, f1*f2*f3*f4*f5*f7*f9, 
  f1*f2*f3*f4*f5*f7*f8*f9, f1*f2*f3*f6*f8*f9, f4*f8*f9, f5*f6*f8*f9, 
  f5*f6, f4*f9, f5, f4*f6*f8, f4*f6*f7*f8*f9, f5*f7, 
  f1*f2*f3*f4*f5*f6*f7*f8, f1*f2*f3, f1*f2*f3*f7*f8, 
  f1*f2*f3*f4*f5*f6, f2*f5*f7*f8, f2*f4*f6*f9, f2*f4*f6*f8, f2*f5*f9, 
  f1*f3*f8*f9, f1*f3*f4*f5*f6, f1*f3, f1*f3*f4*f5*f6*f7*f8, 
  f1*f3*f4*f5*f7*f8*f9, f1*f3*f4*f5*f8*f9, f1*f3*f6*f7*f8, 
  f1*f3*f6*f7*f8*f9, f2*f4*f7*f8*f9, f2*f5*f6*f7*f8, f2*f5*f6, 
  f2*f4*f7*f9, f1*f3*f4*f8, f1*f3*f5*f6*f8*f9, f1*f3*f5*f6, 
  f1*f3*f4*f7, f2*f4*f5*f8*f9, f2*f6*f7*f8*f9, f2*f4*f5*f7, f2*f6*f7, 
  f2*f4*f5*f6*f8, f2*f7, f2*f4*f5*f6*f7, f2*f7*f8, f1*f3*f4*f6*f9, 
  f1*f3*f5*f7*f8*f9, f1*f3*f5*f7*f8, f1*f3*f4*f6*f7*f9, 
  f2*f3*f4*f7*f9, f2*f3*f5*f6*f8, f2*f3*f5*f6*f7*f8, f2*f3*f4*f8, 
  f1*f4*f5*f7*f9, f1*f6*f8, f1*f6*f8*f9, f1*f4*f5*f7*f8, 
  f1*f4*f5*f6*f7*f8, f1*f4*f5*f6*f7*f9, f1*f7*f9, f1*f7, 
  f2*f3*f5*f7*f8, f2*f3*f4*f6*f7*f9, f2*f3*f4*f6*f9, f2*f3*f5*f9, 
  f1*f5*f8*f9, f1*f4*f6, f1*f5*f7, f1*f4*f6*f7*f9, f2*f3*f4*f5*f6*f8, 
  f2*f3*f7*f8, f2*f3*f4*f5*f6*f7*f8*f9, f2*f3, f2*f3*f6*f7, 
  f2*f3*f4*f5*f8, f2*f3*f6*f9, f2*f3*f4*f5*f7, f1*f5*f6*f9, f1*f4*f7, 
  f1*f5*f6*f7*f8, f1*f4*f9, f3*f4*f5*f7, f3*f6*f8, f3*f6*f8*f9, 
  f3*f4*f5*f8*f9, f1*f2*f5*f6*f7, f1*f2*f4*f7*f8, f1*f2*f5*f6*f9, 
  f1*f2*f4*f7, f1*f2*f4*f6*f7*f8*f9, f1*f2*f5*f7*f8*f9, 
  f1*f2*f4*f6*f8*f9, f1*f2*f5*f7, f3*f9, f3*f4*f5*f6*f7*f9, f3*f7*f9, 
  f3*f4*f5*f6*f7*f8*f9, f1*f2*f7*f8, f1*f2*f8, f1*f2*f4*f5*f6*f8, 
  f1*f2*f4*f5*f6*f9, f3*f5*f9, f3*f4*f6*f7, f3*f5*f7*f9, 
  f3*f4*f6*f7*f8, f3*f5*f6, f3*f5*f6*f9, f3*f4*f8, f3*f4*f7*f9, 
  f1*f2*f4*f5*f7, f1*f2*f6*f7*f8*f9, f1*f2*f4*f5*f9, f1*f2*f6*f7*f9 ];;
\end{verbatim}

\end{document}